\documentclass[12pt]{article}

\usepackage[T1]{fontenc}
\usepackage{avant}
\usepackage[cp1250]{inputenc}
\usepackage{amsmath,amssymb,amsthm}
\usepackage{caption}
\usepackage{geometry}
\usepackage{float}

\usepackage{makeidx}
\usepackage[matrix,arrow,curve]{xy}
\usepackage{boxedminipage}
\usepackage{epic}
\usepackage{longtable}
\usepackage{ifthen}
\usepackage{tabularx}
\usepackage{parskip}
\usepackage{verbatim}
\usepackage{placeins}

\title{Clifford-Klein forms and a-hyperbolic rank}
\author{Maciej Boche\'nski and Aleksy Tralle}

\begin{document}

\newtheorem{theorem}{Theorem}
\newtheorem{proposition}{Proposition}
\newtheorem{lemma}{Lemma}
\newtheorem{definition}{Definition}
\newtheorem{example}{Example}
\newtheorem{note}{Note}
\newtheorem{corollary}{Corollary}
\newtheorem{remark}{Remark}

\maketitle{}

\begin{abstract}
The purpose of this article is to introduce and investigate properties of a tool (the a-hyperbolic rank) which enables us to obtain new examples of homogeneous spaces $G/H$ which admit and do not admit a discontinuous action of a non virtually-abelian discrete subgroup. We achieve this goal by exploring in greater detail the technique of adjoint orbits developed by Okuda combined with the well-known conditions of Benoist. We find easy-to-check conditions on $G$ and $H$ expressed directly in terms of the Satake diagrams of the corresponding Lie algebras, in cases when $G$ is a real form of a complex Lie group of type $A_n$, $D_{2k+1}$ or $E_6$.
One of the advantages of this approach is the fact, that we don't need to know the embedding of $H$ into $G.$ Using the a-hyperbolic rank we also show, that the homogeneous space $E_6{}^{\text{IV}}/H$ of reductive type admits a discontinuous action of a non virtually abelian discrete subgroup if and only if $H$ is compact. Also, inspired by the work of Okuda on symmetric spaces $G/H$ we find a list of simply connected $3$-symmetric spaces admitting a discontinuous action of a non virtually-abelian discrete subgroup. This list yields almost complete classification of such spaces (there is one exception). 
\end{abstract}

\section{Introduction}\label{sect:intro}

Let $G/H$ be a homogeneous space of a connected semisimple real linear Lie group $G.$ We are interested in the situation when there exists a discrete and not virtually-abelian subgroup $\Gamma\subset G$ acting discontinuously on $G/H$. We say that $G/H$ admits a compact Clifford-Klein form, if there exists a discrete subgroup $\Gamma\subset G$ acting discontinuously on $G/H$ and with compact quotient $\Gamma\setminus G/H$. It is known that if $G/H$ admits compact Clifford-Klein forms, it necessarily admits discontinuous action of non virtually-abelian discrete subgroup (but not {\it vice versa}). The problem of the existence of Clifford-Klein forms is very important in several research areas. It is not our intention here to describe the whole topic, therefore, we refer to the excellent survey \cite{ko} and papers \cite{kob1, kob2}.     

The purpose of this note is to demonstrate a simple way of checking when certain types of homogenous spaces $G/H$ admit or do not admit discontinuous actions of non virtually-abelian subgroups. Our method is based on a more detailed exploration of the technique of adjoint orbits from \cite{ok}, together with the well-known sufficient condition of Benoist \cite{ben} . It yields  simple  conditions for the existence and the non-existence of such actions expressed in terms of the a-hyperbolic rank, which depends on the Satake diagrams of real Lie algebras $\mathfrak{g}$ and $\mathfrak{h}$. As a consequence we get new examples of homogeneous spaces $G/H$ which admit discontinuous actions of non virtually-abelian discrete subgroups as well as  of $G/H$ which do not. Our main result is stated in Theorem \ref{twg}. Although this theorem gives only sufficient conditions, it resolves the problem of existence of discontinuous actions of non virtually-abelian subgroups in vast classes of examples. As an application we give a list of all simple and connected real 3-symmetric spaces (i.e. homogeneous spaces $G/H$ generated by automorphisms of order 3 of a simple and connected real Lie group $G$) admitting such actions (Table 2) with one exception. Our method does not work for $M=SO(2k+1,2k+1)/(U(1,1)\times SO(2k-1,2k-1))$. Recall that Okuda classified symmetric spaces with this property in \cite{ok}. Generalized symmetric spaces constitute a  large class of homogeneous spaces whose properties are relatively close to the properties of symmetric spaces (see \cite{bt}, \cite{kow}). Therefore, it is natural to extend results from \cite{ok} onto this class. 

Our methods are based on the Lie group theory. We refer to \cite{kn}, \cite{h} and \cite{ov} and use facts from these sources without further explanations. Our notation and terminology is close to \cite{ov}.

\noindent {\bf Acknowledgment.} We would like to especially thank the anonymous referee for very useful suggestions and indicating many inaccuracies. We would like to also thank him/her for pointing out, that one of the assumptions in Theorem \ref{twg} can be dropped. The authors thank Andrei Rapinchuk and Dave Witte-Morris for answering their questions. The second author acknowledges partial support of the ESF Research Network "Contact and Symplectic Topology" (CAST). 

\section{a-hyperbolic rank}

By the a-hyperbolic rank we will understand a dimension of a specific convex cone defined by the action of the Weyl group of a semisimple Lie group $G.$ 

Let $V$ be a real vector space of dimension $n.$ Choose a set of linearly independent vectors $B \subset V.$
A convex cone $A^{+}$ is a subset of $V$ generated by all linear combinations with non-negative coefficients $A^{+}:=Span^{+}(B)$.
The cardinality of $B$ is the dimension of the convex cone $A^{+}.$
In the sequel we will use the simple observation that for any  linear automorphism $f:V \rightarrow V,$ such that $f(A^{+})=A^{+},$  the set of fixed points of $f$ in $A^{+}$ is a convex cone.

\begin{lemma}
Let $V_{1},...,V_{n}$ be a collection of vector subspaces of $V$ and let $A^{+}$ be a convex cone. Assume that
$$A^{+} \subset \bigcup_{k=1}^{n} V_{k}.$$
Then there exists a number $k,$ such that $A^{+} \subset V_{k}.$ 
\label{lin}
\end{lemma}

\subsection{Antipodal hyperbolic orbits}

In this section we are interested in  antipodal hyperbolic orbits in  absolutely simple Lie algebras. Let $G$ be a real, connected and absolutely simple linear Lie group with a Lie algebra $\mathfrak{g}.$ We say that an element $X \in \mathfrak{g}$ is hyperbolic, if $X$ is semisimple (that is, $ad_{X}$ is diagonalizable) and all eigenvalues of $ad_{X}$ are real.

\begin{definition}
{\rm An adjoint orbit $O_{X}:=Ad(G)(X)$ is said to be hyperbolic if $X$ (and therefore every element of $O_{X}$) is hyperbolic. An orbit $O_{Y}$ is antipodal if $-Y\in O_{Y}$ (and therefore for every $Z\in O_{Y},$ $-Z\in O_{Y}$).} 
\end{definition} 

We begin with a brief description of an effective way of classifying antipodal hyperbolic orbits in $\mathfrak{g}^{\mathbb{C}}$ and in $\mathfrak{g}.$ For a more detailed treatment of this subject please refer to \cite{ok}.
\noindent
Fix a Cartan subalgebra $\mathfrak{j}^{\mathbb{C}}$ of $\mathfrak{g}^{\mathbb{C}}.$ Let $\Delta=\Delta (\mathfrak{g}^{\mathbb{C}},\mathfrak{j}^{\mathbb{C}}),$ be the root system of $\mathfrak{g}^{\mathbb{C}}$ with respect to $\mathfrak{j}^{\mathbb{C}}$. Consider the subalgebra
$$\mathfrak{j}:=\{  X\in \mathfrak{j}^{\mathbb{C}} \ | \forall_{\alpha \in \Delta}, \ \alpha (X)\in \mathbb{R} \},$$
which is a real form of $\mathfrak{j}^{\mathbb{C}}.$ Choose a subsystem $\Delta^{+}$ of positive roots in $\Delta.$ Then
$$\mathfrak{j}^{+}:=\{  X\in \mathfrak{j} \ | \forall_{\alpha \in \Delta^{+}}, \ \alpha (X)\geq 0 \}$$
is the closed Weyl chamber for the Weyl group $W_{\mathfrak{g}^{\mathbb{C}}}$ of $\Delta .$

\begin{lemma}[Fact 6.1 in \cite{ok}]
Every hyperbolic orbit in $\mathfrak{g}^{\mathbb{C}}$ meets $\mathfrak{j}$ in a single $W_{\mathfrak{g}^{\mathbb{C}}}$-orbit. In particular there is a bijective correspondence between hyperbolic orbits $O_{X}$ and elements $X$ of $\mathfrak{j}^{+}.$
\label{l1}
\end{lemma}

\noindent
Let $\Pi$ be a simple root system for $\Delta^{+}.$ For every $X \in \mathfrak{j}$ we define
$$\Psi_{X}:\Pi \rightarrow \mathbb{R}, \ \alpha \rightarrow \alpha (X).$$
The above map is called the {\it weighted Dynkin diagram} of $X\in \mathfrak{j},$ and the value $\alpha (X)$ is the weight of the node $\alpha.$ Since $\Pi$ is a base of the dual space $\mathfrak{j}^{\ast}$, the map
$$\Psi : \mathfrak{j} \rightarrow \mathop{\mathrm{Map}}\nolimits (\Pi , \mathbb{R}), \ X \rightarrow \Psi_{X}$$
is an linear isomorphism. We see, that:
$$\Psi |_{\mathfrak{j}^{+}} : \mathfrak{j}^{+} \rightarrow \mathop{\mathrm{Map}}\nolimits(\Pi , \mathbb{R}_{\geq 0}), \ X \rightarrow \Psi_{X}$$
is bijective. 
\noindent
We also need a tool, that will enable us to distinguish antipodal hyperbolic orbits. Let $w_{0}$ be the longest element of $W_{\mathfrak{g}^{\mathbb{C}}}.$ The action of $w_{0}$ sends $\mathfrak{j}^{+}$ to $- \mathfrak{j}^{+}.$ Define:
$$-w_{0}: \mathfrak{j} \rightarrow \mathfrak{j}, \ X \mapsto -(w_{0}X).$$
This is an involutive automorphism of $\mathfrak{j},$ which preserves $\mathfrak{j}^{+}.$ Then $\Psi$ and $-w_{0}$ induce the linear automorphism $\iota = \Psi \circ (-w_{0}) \circ \Psi^{-1}$ of $\mathop{\mathrm{Map}}\nolimits(\Pi, \mathbb{R}).$

\begin{theorem}[Theorem 6.3 in \cite{ok}] Hyperbolic orbit $O$ in $\mathfrak{g}^{\mathbb{C}}$ is antipodal if and only if the weighted Dynkin diagram of $O$ is  invariant with respect to $\iota.$
\label{tw2}
\end{theorem}

\noindent
The involutive automorphism $\iota$ is non-trivial only for $\mathfrak{g}^{\mathbb{C}}$ of type $A_{n}, \ D_{2k+1}, \ E_{6},$ for $n,k \geq 2.$ In such cases the form of $\iota$ is described in Theorem 6.3 (ii) in \cite{ok}.

\noindent
Now we should investigate which hyperbolic orbits in $\mathfrak{g}^{\mathbb{C}}$ meet $\mathfrak{g}.$ Choose a Cartan involution $\theta$ of $\mathfrak{g},$ and the corresponding Cartan decomposition 
$$\mathfrak{g}=\mathfrak{k} + \mathfrak{p}.$$
There exists a Cartan subalgebra $\mathfrak{j}_{\mathfrak{g}}=\mathfrak{t} + \mathfrak{a}$ of $\mathfrak{g}$ for which $\mathfrak{a}$ is a maximal abelian subalgebra of $\mathfrak{p},$ and $\mathfrak{t}$ is in the centralizer $Z_{\mathfrak{k}}(\mathfrak{a})$ (such subalgebra is called split). Notice that
$$\mathfrak{j}^{\mathbb{C}}=\mathfrak{j}_{\mathfrak{g}} + i\mathfrak{j}_{\mathfrak{g}}$$
 and 
$$ \mathfrak{j}=i\mathfrak{t} + \mathfrak{a}.$$
Consider
$$\Sigma := \{  \alpha |_{\mathfrak{a}} \ | \ \alpha \in \Delta \} \setminus \{ 0 \} \subset \mathfrak{a}^*,$$
the restricted root system of $\mathfrak{g}$ with respect to $\mathfrak{a}.$  The set of positive roots has the form
$$\Sigma^{+} := \{  \alpha |_{\mathfrak{a}} \ | \ \alpha \in \Delta^{+} \} \setminus \{ 0 \}.$$
Define
$$\mathfrak{a}^{+} := \{  X\in \mathfrak{a} \ | \forall_{\xi \in \Sigma^{+}}, \ \xi (X) \geq 0 \}.$$
We have $\mathfrak{a}^{+} = \mathfrak{j}^{+} \cap \mathfrak{a}.$ 

\begin{theorem}[Lemma 7.2 and Proposition 4.5 in \cite{ok}] Hyperbolic orbit $O_{X}=Ad(G^{\mathbb{C}})(X)$ in $\mathfrak{g}^{\mathbb{C}}$ meets $\mathfrak{g}$ if and only if the corresponding, unique element $X\in \mathfrak{j}^{+}$ is also in $\mathfrak{a}^{+}.$ Also
$$Ad(G)(X)=O_{X} \cap \mathfrak{g},$$
and every hyperbolic orbit in $\mathfrak{g}$ is obtained this way.
\label{tw3}
\end{theorem}

\noindent
We see that $\iota (\mathfrak{a}^{+})= \mathfrak{a}^{+}$ therefore we can define the 
convex cone

 $$\mathfrak{b}^{+} \subset \mathfrak{a}^{+}$$ 
as the set of all fixed points of $\iota $ in $\mathfrak{a}^{+}.$ 

\noindent
For a semisimple $\mathfrak{g}$ we will take $\mathfrak{b}^{+}$ to be the convex cone spanned by vectors generating convex cones $\mathfrak{b}_{s}^{+}$, defined for every  simple component $\mathfrak{s}$ of $\mathfrak{g}$ (we can perform such construction, since the Cartan involution $\theta$ never maps a nonzero vector  $X$ to a vector orthogonal to $X$ with respect to the  Killing form of the semisimple algebra $\mathfrak{g}$). If $\mathfrak{g}$ is reductive we take $\mathfrak{b}^{+}$ obtained by the described procedure applied to $[\mathfrak{g},\mathfrak{g}]$.

\begin{definition}
{\rm The dimension of $\mathfrak{a}^{+}$ is called the real rank ($\text{rank}_{\mathbb{R}}\mathfrak{g}$) of $\mathfrak{g} .$ The dimension of $\mathfrak{b}^{+}$ is called the a-hyperbolic rank of $\mathfrak{g}$ and is denoted by $\text{rank}_{a-hyp}\mathfrak{g}$.}
\end{definition}

\noindent
By Lemma \ref{l1}, Theorem \ref{tw2} and Theorem \ref{tw3} we have:

\begin{lemma}
There is a bijective correspondence between antipodal hyperbolic orbits $O_{X}$ in $\mathfrak{g}$ and elements $X \in \mathfrak{b}^{+}.$ One also has
$$O_{X}=Ad(G)(X).$$
\label{lma}
\end{lemma}

\noindent

In what follows we will use the notion of the Satake diagram. Let us recall this construction.
Choose  the Cartan decomposition as above. Let $\theta:\frak{g}\rightarrow\frak{g}$ be the Cartan involution defined as before. Recall that all our constructions are performed with respect to $\frak{j}$ and $\frak{j}^{\mathbb{C}}$.  Consider  the complex conjugation $\sigma:\frak{g}^{\mathbb{C}}\rightarrow\frak{g}^{\mathbb{C}}.$  As usual we get the root space decomposition

$$\frak{g}^{\mathbb{C}}=\frak{j}^{\mathbb{C}}+\sum_{\alpha\in\Delta}\frak{g}_{\alpha}.$$
Define the involution $\sigma^*$ on $(\frak{j}^{\mathbb{C}})^*$ by the formula

$$(\sigma^*\varphi)(X)=\overline{\varphi(\sigma(X))},$$
for each $\varphi \in (\mathfrak{j}^{\mathbb{C}})^{\ast}$ and $X\in \frak{j}^{\mathbb{C}}$
If $\alpha\in\Delta$, then $\sigma^*\alpha\in\Delta$, and $\sigma\frak{g}_{\alpha}=\frak{g}_{\sigma^*\alpha}$. Put
$\Delta_0=\{\alpha\in\Delta\,|\,\sigma^*\alpha=-\alpha\}.$
One easily checks that 
$$\Delta_0=\{\alpha\in\Delta\,|\,\alpha|_{\frak{a}}=0\}.$$
Put $\Delta_1=\Delta\setminus\Delta_0.$ A direct check-up shows that 
$\sigma^*(\Delta_0)\subset\Delta_0,$ and $\sigma^*(\Delta_1)=\Delta_1.$
Choose an order in $\Delta$ in a way that $\sigma^*(\Delta_1^+)\subset\Delta_1^+.$ 
 Put $\Pi_0=\Pi\cap\Delta_0$ and $\Pi_1=\Pi\cap\Delta_1$. Recall that the {\it Satake diagram} for $\frak{g}$ is defined as follows. One takes the Dynkin diagram for $\frak{g}^{\mathbb{C}}$ and paints vertexes from $\Pi_0$ in black and vertexes from $\Pi_1$ in white. Next, one shows that $\sigma^*$ determines an involution $\tilde\sigma$ on $\Pi_1$ defined by the equation 
$$\sigma^*\alpha-\beta=\sum_{\gamma\in\Pi_0}k_{\gamma}\gamma,\,k_{\gamma}\geq 0.$$
By definition, if the above equality holds for $\alpha$ and $\beta$, then $\tilde\sigma\alpha=\beta$. Now the construction of the Satake diagram is completed by joining by arrows the white vertexes transformed into each other by $\tilde\sigma$. Recall that semisimple real Lie algebras are uniquely determined by their Satake diagrams up to isomorphism. The table of  the Satake diagrams for all real forms of simple complex Lie algebras can be found in \cite{ov}.

In order to calculate the a-hyperbolic rank we need the following definition and theorems (compare Definition 7.3 in \cite{ok})

\begin{definition}
{\rm Let $\Psi_{X} \in \mathop{\mathrm{Map}}\nolimits(\Pi, \mathbb{R})$ be the weighted Dynkin diagram of $\mathfrak{g}^{\mathbb{C}}$ and $S_{\mathfrak{g}}$ be the Satake diagram of $\mathfrak{g}.$ We say that $\Psi_{X}$ matches $S_{\mathfrak{g}}$ if all black nodes in $S_{\mathfrak{g}}$ have weights equal to 0 in $\Psi_{X}$ and every two nodes joined by an arrow have the same weights.}
\label{dfff}
\end{definition}

\begin{theorem}[Theorem 7.4 in \cite{ok}] The weighted Dynkin diagram $\Psi_{X} \in \mathop{\mathrm{Map}}\nolimits(\Pi , \mathbb{R}_{\geq 0})$ of a hyperbolic orbit $O$ in $\mathfrak{g}^{\mathbb{C}}$ matches $S_{\mathfrak{g}}$ if and only if $O$ meets $\mathfrak{g}.$ There is also a bijective correspondence between elements matching $S_{\mathfrak{g}}$ of $\mathop{\mathrm{Map}}\nolimits(\Pi , \mathbb{R}_{\geq 0})$ and the set of hyperbolic orbits meeting $\mathfrak{g}.$
\label{tw4}
\end{theorem}

\begin{theorem}[Theorem 7.5 in \cite{ok}] The map $\Psi : \mathfrak{j} \rightarrow \mathop{\mathrm{Map}}\nolimits(\Pi , \mathbb{R})$ induces the linear isomorphism:
$$\mathfrak{a} \rightarrow \{  \Psi_{X} \ matches \ S_{\mathfrak{g}} \}, \ \ X \mapsto \Psi_{X}$$ 
\end{theorem}

The procedure of  calculating the a-hyperbolic rank is a straightforward consequence of the cited results and looks as follows.

\noindent \textbf{Step 1.} We calculate the a-hyperbolic rank separately for every simple part of $\mathfrak{g}$ and add results. If $\mathfrak{g}$ is reductive, than to get the a-hyperbolic rank of $\mathfrak{g}$ we take the a-hyperbolic rank of $[\mathfrak{g} ,\mathfrak{g} ].$

\noindent \textbf{Step 2.} We calculate the a-hyperbolic rank for simple $\mathfrak{g}$ ($\text{rank} \mathfrak{g}=n$) by taking the weighted Dynkin diagrams of hyperbolic orbits in $\mathfrak{g}^{\mathbb{C}}$ matching $S_{\mathfrak{g}}$ and preserved by $\iota .$ We interpret weights of a given weighted Dynkin diagram as coordinates of a vector in $\mathbb{R}^{n}.$ All vectors constructed this way give us the convex cone which has dimension equal to $\text{rank}_{a-hyp}\mathfrak{g}.$

\begin{example}
{\rm We will show how to calculate the a-hyperbolic rank of $\mathfrak{g}=\mathfrak{e}_{6}^{\text{IV}}.$ We take the weighted Dynkin diagram of $\mathfrak{e}_{6}^{\mathbb{C}}$ and check how $\iota$ acts on it

$$
\xymatrix@=.4cm{
{\displaystyle \mathop{\circ}^{a}} \ar@{-}[r]<-.8ex>
        & \displaystyle \mathop{\circ}^{b} \ar@{-}[r]<-.8ex>
        & \displaystyle \mathop{\circ}^{c} \ar@{-}[r]<-.8ex> \ar@{-}[d]
        & \displaystyle \mathop{\circ}^{d} \ar@{-}[r]<-.8ex>
        & {\displaystyle \mathop{\circ}^{e}} \\ 
&& {\displaystyle \mathop{\circ}_{f}} &&
} \stackrel{\iota}{\mapsto} \xymatrix@=.4cm{
{\displaystyle \mathop{\circ}^{e}} \ar@{-}[r]<-.8ex>
        & \displaystyle \mathop{\circ}^{d} \ar@{-}[r]<-.8ex>
        & \displaystyle \mathop{\circ}^{c} \ar@{-}[r]<-.8ex> \ar@{-}[d]
        & \displaystyle \mathop{\circ}^{b} \ar@{-}[r]<-.8ex>
        & {\displaystyle \mathop{\circ}^{a}} \\ 
&& {\displaystyle \mathop{\circ}_{f}} &&
}
$$

where $a,b,c,d,e,f \geq 0.$ Next take the Satake diagram of $\mathfrak{e}_{6}^{\text{IV}}$
$$
\xymatrix@=.4cm{
{\displaystyle \mathop{\circ}} \ar@{-}[r]<-.8ex>
        & \displaystyle \mathop{\bullet} \ar@{-}[r]<-.8ex>
        & \displaystyle \mathop{\bullet} \ar@{-}[r]<-.8ex> \ar@{-}[d]
        & \displaystyle \mathop{\bullet} \ar@{-}[r]<-.8ex>
        & {\displaystyle \mathop{\circ}} \\ 
&& {\displaystyle \mathop{\bullet}} &&
}
$$

According to Definition \ref{dfff} the weighted Dynkin diagram of $\mathfrak{e}_{6}^{\mathbb{C}}$ matches the Satake diagram of $\mathfrak{e}_{6}^{\text{IV}}$ if and only if it is of the form:

$$
\xymatrix@=.4cm{
{\displaystyle \mathop{\circ}^{a}} \ar@{-}[r]<-.8ex>
        & \displaystyle \mathop{\circ}^{0} \ar@{-}[r]<-.8ex>
        & \displaystyle \mathop{\circ}^{0} \ar@{-}[r]<-.8ex> \ar@{-}[d]
        & \displaystyle \mathop{\circ}^{0} \ar@{-}[r]<-.8ex>
        & {\displaystyle \mathop{\circ}^{e}} \\ 
&& {\displaystyle \mathop{\circ}_{0}} &&
}
$$

Moreover, the above diagram is preserved by $\iota$ if and only if $a=e.$ Hence we obtain the following weighted Dynkin diagram:

$$
\xymatrix@=.4cm{
{\displaystyle \mathop{\circ}^{a}} \ar@{-}[r]<-.8ex>
        & \displaystyle \mathop{\circ}^{0} \ar@{-}[r]<-.8ex>
        & \displaystyle \mathop{\circ}^{0} \ar@{-}[r]<-.8ex> \ar@{-}[d]
        & \displaystyle \mathop{\circ}^{0} \ar@{-}[r]<-.8ex>
        & {\displaystyle \mathop{\circ}^{a}} \\ 
&& {\displaystyle \mathop{\circ}_{0}} &&
}
$$

Therefore $\mathfrak{b}^{+}$ has a dimension equal to the dimension of $Span^{+}((1,0,0,0,1,0)).$ Thus $\text{rank}_{a-hyp}\mathfrak{e}_{6}^{\text{IV}}=1$.}
 
\end{example}

Using the described procedure we get the following table of a-hyperbolic ranks which are not equal to the real ranks of the corresponding Lie algebras.
\FloatBarrier

\begin{center}
 \begin{table}[ht]
 \centering
 {\footnotesize
 \begin{tabular}{| c | c | c|}
   \hline
   \multicolumn{3}{|c|}{\textbf{Table 1.} \textbf{\textit{A-HYPERBOLIC RANK}}} \\
   \hline                        
   $\mathfrak{g}$ & \textbf{a-hyperbolic rank} & $\text{rank}_{\mathbb{R}}\mathfrak{g}$ \\
   \hline
   $\mathfrak{sl}(2k,\mathbb{R})$  & k &  2k-1 \\
   \footnotesize $k\geq 1$ & & \\
   \hline
   $\mathfrak{sl}(2k+1,\mathbb{R})$  & k & 2k \\
   \footnotesize $k\geq 1$ & & \\
   \hline
   $\mathfrak{su}^{\ast}(4k)$  & k & 2k-1 \\
   \footnotesize $k\geq 1$ & & \\
   \hline
   $\mathfrak{su}^{\ast}(4k+2)$  & k & 2k \\
   \footnotesize $k\geq 1$ & & \\
   \hline
   $\mathfrak{so}(2k+1,2k+1)$  & 2k & 2k+1 \\
   \footnotesize $k\geq 2$ & &  \\
   \hline
	 $\mathfrak{e}_{6}^{\text{I}}$ & 4 & 6 \\
	 \hline
   $\mathfrak{e}_{6}^{\text{IV}}$  & 1 & 2 \\
   \hline  
 \end{tabular}
 }
\captionsetup{justification=centering}
 \caption{
 This table contains all real forms of simple Lie algebras $\mathfrak{g}^{\mathbb{C}},$ for which $\text{rank}_{\mathbb{R}}\mathfrak{g} \neq \text{rank}_{a-hyp}\mathfrak{g}.$
 }
 \label{tab1}
 \end{table}
\end{center}

\noindent
We also need the following fact.

\begin{theorem}[Facts 5.1 and 5.3 in \cite{ok}]
Let $O$ be an antipodal hyperbolic orbit in $\mathfrak{g}$ and $W_{\mathfrak{g}}$ be the Weyl group of $\mathfrak{g}.$ Then $O \cap \mathfrak{a}$ is a single $W_{\mathfrak{g}}$-orbit in $\mathfrak{a}.$
\label{ttt1}   
\end{theorem}

\subsection{Discontinuous actions of non virtually-abelian discrete subgroups}

Let $G$ be a linear, connected, semisimple, real Lie group with Lie algebra $\mathfrak{g}$ and $H\subset G$ be a closed and connected subgroup. Also, let $\mathfrak{h}$ be the Lie algebra of $H$ and $W_{\mathfrak{g}}$ be the Weyl group of $\mathfrak{g}.$

\begin{definition}
{\rm The subgroup $H$ is reductive in $G$ if $\mathfrak{h}$ is reductive in $\mathfrak{g},$ that is, there exists a Cartan involution $\theta $ for which $\theta (\mathfrak{h}) = \mathfrak{h}.$
The space $G/H$ is called the homogeneous space of reductive type. Moreover, in this setting the Lie algebra $\mathfrak{h}$ is reductive (which means that it is a sum of its center and the  semisimple part $[\mathfrak{h},\mathfrak{h}]$)}.
\end{definition}

\noindent
The Lie subalgebra $\mathfrak{h}$ is reductive in $\mathfrak{g},$ therefore we can choose a Cartan involution $\theta$ of $\mathfrak{g}$ preserving $\mathfrak{h}.$ We obtain the Iwasawa decomposition
$$\mathfrak{h} = \mathfrak{k}_{h} + \mathfrak{a}_{h} + \mathfrak{n}_{h}$$
which is compatible with the decomposition $\mathfrak{g} = \mathfrak{k} + \mathfrak{a} + \mathfrak{n}$ (that is $\mathfrak{k}_{h} \subset \mathfrak{k},$ $\mathfrak{a}_{h} \subset \mathfrak{a}$ and $\mathfrak{n}_{h} \subset \mathfrak{n}$). Y. Benoist in \cite{ben} gave the following characterization of homogeneous spaces $G/H$ which admit a discontinuous action of a non virtually abelian subgroup $\Gamma\subset G$.

\begin{theorem}[Theorem 1 in \cite{ben}]
The group $G$ contains a discrete and non virtually abelian subgroup $\Gamma$ which acts discontinuously on $G/H$ if and only if for every $w$ in $W_{\mathfrak{g}}$, $w \cdot \mathfrak{a}_{h}$ does not contain $\mathfrak{b}^{+}.$ Also one can choose $\Gamma$ to be Zariski dense in $G$.
\label{be}
\end{theorem}

\noindent
In this section we will show an effective way of determining if such action exists. Let $\mathfrak{g}^{\mathbb{C}}$ be a complexification of $\mathfrak{g}$ and assume that a subalgebra $\mathfrak{h} \subset \mathfrak{g}$ is reductive in $\mathfrak{g}.$ Since the subalgebra $\mathfrak{h}$ is reductive we have
$$\mathfrak{h} = \mathfrak{z} (\mathfrak{h}) + [\mathfrak{h},\mathfrak{h}],$$

\noindent
Let $\mathfrak{b}_{[h,h]}^{+}$ be the convex cone constructed according to the procedure described in the previous subsection (for  $[\mathfrak{h},\mathfrak{h}]$). We will also need the following lemma.

\begin{lemma}
Let $X\in \mathfrak{b}_{[h,h]}^{+}.$ The orbit $O_{X}^{g}:=Ad(G)(X)$ is an antipodal hyperbolic orbit in $\mathfrak{g}.$
\label{lemm2}
\end{lemma}

\begin{proof}
The vector $X$ defines an antipodal hyperbolic orbit in $\mathfrak{h}.$ Therefore we can find $h \in H \subset G$ such that
$Ad_{h}(X) = - X$.
Since for the decomposition (defined by the Cartan involution $\theta$) 
$$\mathfrak{g} = \mathfrak{k} + \mathfrak{a} + \mathfrak{n}$$ 
the space $\mathfrak{a}$ consists of vectors for which $ad$ is diagonalizable with real values and
$X \in \mathfrak{b}_{[h,h]}^{+} \subset \mathfrak{a}_{h} \subset \mathfrak{a}$,
 therefore, vector $X$ is hyperbolic in $\mathfrak{g}.$ 
Thus $Ad(G)(X)$ is a hyperbolic orbit in $\mathfrak{g}$ and $-X \in Ad(G)(X).$
\end{proof}

\subsection{Main result}
Now we are ready to state the main result of this work. It can be considered as an important supplement to the celebrated Kobayashi's theorem concerning discontinuous actions of  discrete subgroups. The latter is cited below (recall once more that we assume all Lie groups to be linear).

\begin{theorem}[Corollary 4.4 in \cite{kob1}]
For a homogeneous space $G/H$ of reductive type, the following conditions are equivalent:
\begin{enumerate}
\item $G/H$ admits a  discontinuous action of an infinite discrete subgroup $\Gamma$ of $G$;
\item $\text{\rm rank}_{\mathbb{R}}\mathfrak{g} > \text{\rm rank}_{\mathbb{R}}\mathfrak{h}.$
\end{enumerate}
\end{theorem}

Our main result is the following  theorem.

\begin{theorem}\label{twg}
Let $G$ be a connected and semisimple linear Lie group and $H$ a reductive subgroup with a finite number of connected components. Let $\mathfrak{g}$ and $\mathfrak{h}$ denote the appropriate Lie algebras. Then
 \begin{enumerate}
  \item If $\text{\rm rank}_{a-hyp}\mathfrak{g} = \text{\rm rank}_{a-hyp}\mathfrak{h}$ then $G/H$ does not admit discontinuous actions of non virtually-abelian discrete subgroups (and, therefore,  compact Clifford-Klein forms).
  \item If $\text{\rm rank}_{a-hyp}\mathfrak{g} > \text{\rm rank}_{\mathbb{R}} \mathfrak{h}$ then $G/H$ admits a discontinuous action of a non virtually-abelian discrete subgroup.
 \end{enumerate}
\end{theorem}

\begin{proof}
Let us begin with the proof of the first claim of the theorem. By Theorem \ref{be}, the non-existence of a discontinuous action of a non virtually-abelian discrete subgroup for $G/H$ can be reformulated as the following sequence of the equivalent conditions:
\begin{enumerate}
\item $\mathfrak{b}^{+} \subset w\cdot \mathfrak{a}_{h}$ for some $w\in W_{\mathfrak{g}},$
\item $Span(\mathfrak{b}^{+}) \subset w\cdot \mathfrak{a}_{h}$ for some $w\in W_{\mathfrak{g}},$
\item $w^{-1}\cdot Span(\mathfrak{b}^{+}) \subset \mathfrak{a}_{h}$ for some $w\in W_{\mathfrak{g}}.$
\end{enumerate} 
Here is the proof of their equivalence. The first two are equivalent, because  if a vector space contains $B$ then it also contains $Span(B).$ The equivalence of the second and the third condition is straightforward.

Taking into consideration the above equivalences, one reformulates  the condition in Theorem \ref{be} as follows: for every $w$ in $W_{\mathfrak{g}}$ the space $w\cdot Span(\mathfrak{b}^{+})$ is not contained in $\mathfrak{a}_{h}.$
According to Lemma \ref{lma} and Lemma \ref{lemm2} the  orbit $Ad(G)(X)$ for $X\in \mathfrak{b}_{[h,h]}^{+} \subset \mathfrak{a}$ is an antipodal hyperbolic orbit. Therefore we can find $Y \in \mathfrak{b}^{+} \subset \mathfrak{a}$ such that
$Ad(G)(X)=Ad(G)(Y).$
Theorem \ref{ttt1} implies that $X,Y$ are in a certain $W_{\mathfrak{g}}$-orbit in $\mathfrak{a},$ that is $X=w\cdot Y$ for some $w\in W_{\mathfrak{g}}.$ We shall prove that such $w\in W_{\mathfrak{g}}$ can
be taken independently from the choice of $X \in \mathfrak{b}_{[h,h]}^{+}$. We have
$$\mathfrak{b}_{[h,h]}^{+} \subset W_{\mathfrak{g}} \cdot \mathfrak{b}^{+} = \bigcup_{w\in W_{\mathfrak{g}}} w\cdot \mathfrak{b}^{+}.$$
Since $\mathfrak{b}^{+} \subset Span(\mathfrak{b}^{+})$ we obtain
$$\mathfrak{b}_{[h,h]}^{+} \subset W_{\mathfrak{g}} \cdot Span(\mathfrak{b}^{+}) \subset \mathfrak{a}.$$
Taking into consideration Lemma \ref{lin} we get the inclusion
\begin{equation}
\mathfrak{b}_{[h,h]}^{+} \subset w\cdot Span(\mathfrak{b}^{+}),
\label{eqq1}
\end{equation}
for some $w \in W_{\mathfrak{g}}.$ Since the a-hyperbolic ranks are equal one has the equality 
$$\dim\mathfrak{b}_{[h,h]}^{+}=\dim \ Span(\mathfrak{b}^{+}).$$ 
Since $w$ is an automorphism one obtains
$$Span(\mathfrak{b}_{[h,h]}^{+})=w\cdot Span(\mathfrak{b}^{+}),$$
which implies
$$w\cdot Span(\mathfrak{b}^{+})=Span(\mathfrak{b}_{[h,h]}^{+}) \subset \mathfrak{a}_{h}.$$

\noindent
The second claim is also straightforward. Every element $w\in W_{\mathfrak{g}}$ acts on $\mathfrak{a}$ by linear transformations, and, therefore, preserves the dimension $n$ of the subspace $\mathfrak{a}_{h} \subset \mathfrak{a}.$ Our assumption implies that $\mathfrak{b}^{+}$ contains subset of more than $n$ linearly independent vectors, therefore $\mathfrak{b}^{+}$ can not be a subset of linear subspace $w\cdot \mathfrak{a}_{h} \subset \mathfrak{a}$ (for any $w$).

\end{proof} 

The following remarks seem to be in order. The celebrated Calabi-Marcus phenomenon settled by Kobayashi states that $G/H$ does not admit infinite discontinuous group if and only of 
\begin{itemize}
\item  $\text{rank}_{\mathbb{R}}G=\text{rank}_{\mathbb{R}}H \ \ ({\text{A}})$.
\end{itemize}
Theorem \ref{twg} states that $G/H$ does not admit non virtually abelian discontinuous group if 
\begin{itemize}
\item  $\text{rank}_{a-hyp}G=\text{rank}_{a-hyp}H \ \ ({\text{B}})$.
\end{itemize} 
Conversely, non virtually abelian discontinuous groups exist for $G/H$ if 
\begin{itemize}
\item  $\text{rank}_{a-hyp}G>\text{rank}_{\mathbb{R}}H \ \ (\text{C})$. 
\end{itemize}
Thus, it would be interesting to understand to what extent conditions (A),(B) and (C) together are close to a criterion. Notice that the first condition in Theorem \ref{twg} (that is, condition (B)) is not necessary for the non-existence of discontinuous groups that are non virtually abelian. Consider the following examples.
\begin{example}
{\rm For $G=SL(10,\mathbb{R})$ and $H=SO(5,5)$ we have
$$\text{rank}_{a-hyp}\mathfrak{g} = 5 > 4 = \text{rank}_{a-hyp}\mathfrak{h},$$
and also $\text{rank}_{\mathbb{R}}G=9>\text{rank}_{\mathbb{R}}H=5$.
By the results of Benoist (see Example 1 in \cite{ben}) $G/H$ does not admit actions of non virtually-abelian discontinuous groups, but none of the conditions (A) or (B) is satisfied.}
\end{example}

\begin{example}
{\rm For $G=SL(10,\mathbb{R})$ and $H=SL(3,\mathbb{R}) \times SL(7,\mathbb{R})$ we have
$$\text{rank}_{a-hyp}\mathfrak{g} = 5 > 4 = 1+3 = \text{rank}_{a-hyp}\mathfrak{h},$$
and $G/H$ admits actions of non virtually-abelian discontinuous groups (see Example 2 in \cite{ben}). Notice that the condition (C) is not satisfied since $\text{rank}_{\mathbb{R}}\mathfrak{h}=8$}.
\end{example}

\section{New examples}

Let $G$ be a semisimple, linear Lie group with Lie algebra $\mathfrak{g}$ and $H\subset G$ a closed subgroup with Lie algebra $\mathfrak{h}$.

\begin{theorem}[\cite{yo}]
If $H$ is semisimple then $H$ is reductive in $G.$
\label{te1}
\end{theorem}

\begin{corollary}
If $\text{\rm rank}_{a-hyp}\mathfrak{h} = \text{\rm rank}_{a-hyp}\mathfrak{g}$ for semisimple $H$ then $G/H$ does not admit a discontinuous action of a non virtually-abelian discrete subgroup. If $\text{\rm rank}_{\mathbb{R}}\mathfrak{h} < \text{\rm rank}_{a-hyp}\mathfrak{g}$ then $G/H$ admits a discontinuous action of a non virtually-abelian discrete subgroup.
\end{corollary}

\begin{theorem}[\cite{mos}]
If $G_{n} \subset G_{n-1} \subset ... \subset G_{0}$ and $G_{i}/G_{i-1}$ is of reductive type (for $1\leq i \leq n$) then $G_{0}/G_{n}$ is a homogeneous space of reductive type.
\end{theorem}

\begin{corollary}
If $0 < \text{\rm rank}_{a-hyp}\mathfrak{g}_{n} = \text{\rm rank}_{a-hyp}\mathfrak{g}_{0},$ then $G_{i}/G_{j}$ does not admit compact Clifford-Klein forms for $i<j$. If $\text{\rm rank}_{\mathbb{R}}\mathfrak{g}_{j} < \text{\rm rank}_{a-hyp}\mathfrak{g}_{i}$ then $G_{i}/G_{j}$ admits a discontinuous action of a non virtually-abelian discrete subgroup.
\end{corollary}
The following examples are obtained by calculating the a-hyperbolic ranks of the corresponding $\mathfrak{g}$ and $\mathfrak{h}$ (according to Table 1). 

\begin{example}
{\rm The following homogeneous spaces do not admit compact Clifford-Klein forms:}
$$SL(4k+2l,\mathbb{R})/SO(2k,2k)\times Sp(l,\mathbb{R});$$ 
$$SL(2k+2l,\mathbb{R})/Sp(k,\mathbb{R})\times Sp(l,\mathbb{R});$$ 
$$SL(4k+4l,\mathbb{R})/SO(2k,2k)\times SO(2l,2l);$$ 
$$SL(4k+2l+1,\mathbb{R})/SO(2k,2k)\times SO(l,l+1);$$
$$SU^{\ast}(4k+2)/U(s,r-s)\times Sp(t,2k+1-r-t), \ \text{for} \ s+t=k+1, \ 1 \leq r \leq 2k+1;$$
$$SU^{\ast}(4k)/U(s,r-s)\times Sp(t,2k+1-r-t), \ \text{for} \ s+t=k, \ 1 \leq r \leq 2k.$$
\end{example}

\begin{example}
{\rm The following homogeneous spaces admit a discontinuous action of a non virtually-abelian discrete subgroup:}
$$SL(2k+2l+2,\mathbb{R})/SO(k,k+1)\times SO(l,l+1);$$ 
$$SL(2k+2l+2,\mathbb{R})/SO(k,k)\times SO(l,l);$$
$$ E_6^{{\rm I}} / \{ SL(3,\mathbb{C}) \times SU(2,1) \} / \mathbb{Z}_{3} $$
{\rm Note that we use the notation $E_6^{\text{I}}$ for the simply connected real Lie group corresponding to the Lie algebra $\mathfrak{e}_6^{\text{I}}$}. 
\end{example}

\noindent
The a-hyperbolic rank gives us also an easy way of determining if a subalgebra $\mathfrak{h} \subset \mathfrak{g}$ can determine a closed, reductive subgroup $H$ in $G.$

\begin{lemma}
If $H$ is closed and reductive subgroup of $G$  then
$$\text{\rm rank}_{a-hyp}\mathfrak{h}\leq \text{\rm rank}_{a-hyp}\mathfrak{g}.$$
In particular, if $H$ is closed and semisimple subgroup of $G$ then the above condition has to be satisfied.
\end{lemma}

\begin{proof}
The Lemma follows from equation (\ref{eqq1}) (in the proof of Theorem \ref{twg}) and Theorem \ref{te1}.
\end{proof}

Let $E_{6}^{\text{IV}}$ be a simply connected real Lie group corresponding to $\mathfrak{e}_{6}^{\text{IV}}$
\begin{example}
{\rm Lie group $E_{6}^{\text{IV}}$ does not admit $G_{2},$ $SO(2,3),$ $SO(2,5),$ $SO(2,7)$ and $Sp(2,\mathbb{R})$ as closed subgroups.}
\end{example}

\noindent
We also have the following theorem.

\begin{theorem}
Assume that $G=E^{\text{\rm IV}}_6,SO^{\ast}(6),SL(3,\mathbb{R})$ and $H$ is a non-compact subgroup of reductive type. Then $G/H$ does not admit compact Clifford-Klein forms.
\end{theorem}

\begin{proof}
Note that $\text{rank}_{a-hyp}\mathfrak{g}=1.$ If $\text{rank}_{a-hyp}\mathfrak{h}=1$ then $G/H$ does not admit Clifford-Klein forms. Consider the case $\text{rank}_{a-hyp}\mathfrak{h}=0$. If $G/H$ admitted compact Clifford-Klein forms, then $H$ would have compact center by Corollary 1 in \cite{benl}. But the latter, together with vanishing $\text{rank}_{a-hyp}\mathfrak{h}$ would imply $\text{rank}_{\mathbb{R}}\mathfrak{h}=0$. Hence, $H$ would be compact, a contradiction. 
\end{proof}

\noindent
Note that this property was already known for $G=SL(3,\mathbb{R})$ (Prop. 1.10 in \cite{ow}).
\vskip6pt
Now we mention an observation with possible applications to symplectic topology. It is based on the following result. 
\begin{theorem}[\cite{kob}]
If $X \in \mathfrak{g}$ is a semisimple element, then the semisimple orbit $G/Z_{G}(X) \cong Ad(G)(X)$ is a homogeneous space of reductive type, where 
$$Z_{G}(X) := \{ g \in G \ | \ Ad(g)X = X \}.$$
\end{theorem}

\noindent
The above property shows us an interesting method of checking if a given elliptic orbit can be compactified.

\begin{example}
{\rm Every elliptic orbit of $SL(4, \mathbb{R})$ with a non-compact isotropy subgroup (with compact center) admits a discontinuous action of a non virtually-abelian discrete subgroup (please refer to \cite{bou} for a classification of the isotropy subgroups of elliptic orbits). }
\end{example}

\section{$3$-Symmetric spaces}

\begin{definition}
{\rm  A $k$-symmetric space is a triple $(G,H,\sigma)$, where $G$ is a connected Lie group, $H \subseteq G$ is a closed subgroup, and $\sigma :G \rightarrow G$ is an automorphism of G such that:}
\begin{enumerate}
	\item {\rm $\sigma^{k} =id$ and $k\geq2$ is the least integer with this property,
	\item ${(G^{\sigma})}^{o} \subset H \subset G^{\sigma}$ , where $G^{\sigma} =\{ g\in G \mid \sigma (g)=g \}$ , and ${(G^{\sigma})}^{o}$ is the identity component of $G^{\sigma}$.}
\end{enumerate}
\label{dd}
\end{definition}
Clearly, if $k=2$ we get a class of symmetric spaces, that is, homogeneous spaces generated by involutive automorphisms. For the general theory we refer to \cite{kow}.
In this article we consider the case $k=3.$ Simply connected 3-symmetric spaces with an effective action of $G$ where classified by J. Wolf and A. Gray in \cite{g1} and \cite{g2}. Here is their classification result (contained in Tables 7.11-7.14 in the cited paper).

\begin{theorem}[Wolf and Gray]\label{thm:w-g} Let $G/H$ be a simply connected coset space where $G$ is a connected Lie group acting effectively. Suppose $\mathfrak{h}=\mathfrak{g}^{\sigma}$ where $\sigma$ is an automorphism of order 3 on $\mathfrak{g}$ which does not preserve any proper ideals. Then $G$ is reductive, $H$ is a closed reductive subgroup and Tables 7.11-7.14 in \cite{g2} give a complete list of the possibilities (up to isomorphism).
\end{theorem}
\begin{remark} {\rm Note that in \cite{g2} the authors use $G^*$ and $K^*$ instead of $G$ and $H$.}
\end{remark}

As an application of Theorem \ref{twg} we  give a list of simply connected and simple (i.e. $G$ is a simple Lie group) real 3-symmetric spaces $M=G/H$ with an effective action of $G$ admitting discontinuous action of some non virtually-abelian, discrete subgroup $\Gamma \subset G.$ The table below lists all such spaces with only one exception:
$$M=SO(2k+1,2k+1)/(U(1,1)\times SO(2k-1,2k-1))$$
since this space does not fulfill any condition from Theorem \ref{twg}. In this paper we will not study if $M$ admits or not non virtually-abelian discontinuous groups.
To obtain this list we take the classification of (non-compact) simply connected, $3$-symmetric spaces with an effective action of $G$ obtained in Theorem \ref{thm:w-g}, choose the corresponding pairs from Tables 7.11-7.14 in \cite{g2} and check case by case the a-hyperbolic ranks and real ranks of $\mathfrak{g}$ and $\mathfrak{h}$ of every space listed there with simple $G.$ We complete our classification using Theorem \ref{twg}.

\FloatBarrier
 \begin{table}[ht]
 \centering
 {\footnotesize
 \begin{tabular}{| c | c |}
   \hline
   \multicolumn{2}{|c|}{Table 2. Non-compact, simple 3-symmetric spaces admitting a discontinuous} \\
	\multicolumn{2}{|c|}{action of a non virtually-abelian discrete subgroup.} \\
   \hline  
   G & H \\
   \hline 
   $SL(2n,\mathbb{R})/\mathbb{Z}_{2}$ & $\{ SL(n,\mathbb{C})\times T^{1} \}/\mathbb{Z}_{n}$ \\
   \hline
   $SO(2n+1-2s-2t,2s+2t)$ & $U(a-s,s)\times SO(2n-2a+1-2t,2t)$ \\
     & \footnotesize  $1\leq a\leq n, \ 2\leq 2s\leq a$ \\
   \hline
   $Sp(n,\mathbb{R})/\mathbb{Z}_{2}$ & $\{ U(a-s,s)\times Sp(n-a, \mathbb{R}) \}/\mathbb{Z}_{2}$ \\
     & \footnotesize $1\leq a \leq n, \ 2\leq 2s \leq a$ \\
   \hline
   $SO(2n-2s-2t,2s+2t)/\mathbb{Z}_{2}$ & $\{ U(a-s,s)\times SO(2n-2a-2t,2t) \}\mathbb{Z}_{2}$ \\
     & \footnotesize $1\leq a \leq n, \ 0\leq 2s \leq a, \ 0\leq 2t \leq n-a, (s,t)\neq (0,0) $\\
   \hline
   $SO^{\ast}(2n)/\mathbb{Z}_{2}$ & $\{ U(a-s,s)\times SO^{\ast}(2n-2a) \}\mathbb{Z}_{2}$ \\
     & \footnotesize $1\leq a \leq n, \ 0\leq 2s \leq a$ \\
   \hline
     $G_{2}$ & $U(1,1), \ SU(2,1)$ \\
   \hline
   $F_{4}^{\text{I}}$ & $\{ Spin(7-r,r)\times T^{1} \}/\mathbb{Z}_{2}, \ r=2,3$ \\
     & $\{ Sp(3,\mathbb{R})\times T^{1} \}/\mathbb{Z}_{2}, \ \{ Sp(2,1)\times T^{1} \}/\mathbb{Z}_{2}$ \\
     & $\{ SU(3)\times SU(2,1) \} / \mathbb{Z}_{3} \ \{ SU(2,1)\times SU(2,1) \} / \mathbb{Z}_{3}$ \\
		\hline
	
	\end{tabular}
 }
 \label{tab23}
 \end{table}

 \begin{table}
 \centering
 {\footnotesize
 \begin{tabular}{| c | c |}
   \hline
   \multicolumn{2}{|c|}{Table 2. continuation} \\

   \hline  
	 G & H \\
   \hline

	 $E_{6}^{\text{I}}$ & $\{ SL(3,\mathbb{C}) \times SU(2,1) \} / \mathbb{Z}_{3} $ \\
	 \hline
   $E_{6}^{\text{II}}$ & $\{ SO^{\ast}(10)\times SO(2) \}/\mathbb{Z}_{2}$ \\
     & $\{ S(U(5-p,p)\times U(1))\times SU(2-s,s) \}/\mathbb{Z}_{2},$ \\
     & $(s,p)=(0,1),(0,2),(1,2)$ \\
     & $\{ [SU(6-p,p)/\mathbb{Z}_{3}]\times T^{1}  \}/\mathbb{Z}_{2}, \ p=0,2,3$ \\
     & $\{ [SO^{\ast}(8)\times SO(2)]\times SO(2) \}/\mathbb{Z}_{2}$ \\
     & $\{ [SO(6,2)\times SO(2)]\times SO(2) \}/\mathbb{Z}_{2}$ \\
		 & $\{ SU(2,1) \times SU(3)\times SU(3) \} / \{ \mathbb{Z}_{2} \times \mathbb{Z}_{3} \}$ \\  
	   &  $\{ SU(2,1) \times SU(2,1)\times SU(2,1) \} / \{ \mathbb{Z}_{2} \times \mathbb{Z}_{3} \}$ \\
   \hline
   $E_{6}^{\text{III}}$ & $\{ S(U(5-p,p)\times U(1))\times SU(2-s,s) \}/\mathbb{Z}_{2},$ \\    
     & $(s,p)=(1,0),(0,1),(1,1)$ \\
     & $\{ [SU(5,1)/\mathbb{Z}_{3}]\times T^{1} \}/\mathbb{Z}_{2}$ \\
   \hline
   $E_{7}^{\text{V}}$ & $\{ E_{6}^{\text{II}}\times T^{1} \}/\mathbb{Z}_{2}, \ \{ SU(2)\times [SO^{\ast}(10)\times SO(2)] \}/\mathbb{Z}_{2}$ \\
     & $\{ SU(1,1)\times [SO(6,4)\times SO(2)] \}/\mathbb{Z}_{2}$ \\
     & $\{ SO(2)\times SO^{\ast}(12) \}/\mathbb{Z}_{2}, \  \{ SO(2)\times SO(6,6) \}/\mathbb{Z}_{2}$ \\
     & $S(U(4,3)\times U(1))/\mathbb{Z}_{4}$ \\
		 & $\{ SU(3)\times [SU(5,1)/ \mathbb{Z}_{2}]/ \mathbb{Z}_{3}, \ \{ SU(2,1)\times [SU(3,3)/ \mathbb{Z}_{2}]/ \mathbb{Z}_{3}$ \\
   \hline
   $E_{7}^{\text{VI}}$ & $\{ E_{6}^{\text{III}}\times T^{1} \}/\mathbb{Z}_{2}$ \\
     & $\{ SU(2-p,p)\times [SO(10-s,s)\times SO(2)] \}/\mathbb{Z}_{2}$  \\
		 &  $(p,s)=(0,2),(1,2)$ \\
     & $\{ SU(1,1)\times [SO^{\ast}(10)\times SO(2)] \}/\mathbb{Z}_{2}$ \\
     & $S(U(7-s,s)\times U(1))/\mathbb{Z}_{4}, \ s=1,2,3$ \\
		 & $\{ SU(2,1)\times [SU(6)/ \mathbb{Z}_{2}]/ \mathbb{Z}_{3}, \ \{ SU(3)\times [SU(4,2)/ \mathbb{Z}_{2}]/ \mathbb{Z}_{3}$ \\
	   & $\{ SU(2,1)\times [SU(4,2)/ \mathbb{Z}_{2}]/ \mathbb{Z}_{3}$ \\
   \hline
   $E_{7}^{\text{VII}}$ & $ \{ E_{6}^{\text{III}} \times T^{1} \}/\mathbb{Z}_{2}$ \\
    & $\{ SU(1,1)\times [SO(10)\times SO(2)] \}/\mathbb{Z}_{2}$ \\ 
		&  $\{ SU(2)\times [SO^{\ast}(10)\times SO(2)] \}/\mathbb{Z}_{2}$ \\
    & $ \{ SO(2)\times SO(10,2) \}/\mathbb{Z}_{2}$ \\ 
		&  $S(U(7-s,s)\times U(1))/\mathbb{Z}_{4}, \ s=1,2$ \\
		& $\{ SU(2,1)\times [SU(5,1)/ \mathbb{Z}_{2}]/ \mathbb{Z}_{3}$ \\
   \hline
   $E_{8}^{\text{VIII}}$ & $SO(8,6)\times SO(2), \ SO^{\ast}(14)\times SO(2)$ \\
     & $\{ E_{7}^{\text{VI}}\times T^{1} \}/\mathbb{Z}_{2}, \ \{ E_{7}^{V}\times T^{1} \}/\mathbb{Z}_{2}$ \\
		 & $\{ SU(3)\times E_{6}^{\text{III}}  \} / \mathbb{Z}_{3}, \ \{ SU(2,1)\times E_{6}^{\text{II}}  \} / \mathbb{Z}_{3}$ \\
	   & $\{ SU(8,1)  \} / \mathbb{Z}_{3}, \ \{ SU(5,4) \} / \mathbb{Z}_{3}$ \\
   \hline
   $E_{8}^{\text{IX}}$ & $SO(12,2)\times SO(2), \ SO^{\ast}(14)\times SO(2) \ \{ E_{7}^{\text{VII}}\times T^{1} \}/\mathbb{Z}_{2}$ \\
     & $\{ SU(2,1)\times E_{6}  \} / \mathbb{Z}_{3}, \ \{ SU(2,1)\times E_{6}^{III}  \} / \mathbb{Z}_{3}$ \\
	   & $\{ SU(3)\times E_{6}^{\text{II}}  \} / \mathbb{Z}_{3}, \ \{ SU(7,2)  \} / \mathbb{Z}_{3}, \ \{ SU(6,3)  \} / \mathbb{Z}_{3}$ \\
	\hline
	 $SO(4,4)$ & $\{ SU(2,1) \} / \mathbb{Z}_{3}$ \\
	\hline
	 $Spin(5,3)$ & $G_{2}$ \\
	\hline
	 $Spin(4,4)$ & $G_{2}$ \\
	\hline \hline
 \end{tabular}
 }
 \label{tab23}
 \end{table}

Department of Mathematics and Computer Science

University of Warmia and Mazury

S\l\/oneczna 54, 10-710, Olsztyn, Poland

e-mail: mabo@matman.uwm.edu.pl (MB), 

tralle@matman.uwm.edu.pl (AT)


\newpage

\begin{thebibliography}{99}

\bibitem{ben} Y. Benoist,  {\it Actions propres sur les espaces homogenes reductifs},  Ann. of Math. 144 (1996),  315-347.

\bibitem{benl} Y. Benoist, F. Labourie, {\it Sur les espaces homogenes modeles de varietes compactes},  Publications Mathematiques de I.H.E.S 76 (1992) 99-109.

\bibitem{bt} M. Boche\'nski, A. Tralle,  {\it Generalized symplectic symmetric spaces},  Geom. Dedicata, 2013, DOI 10.1007/s10711-013-9902-x 

\bibitem{bou} N. Boumuki,  {\it Isotropy subalgebras of elliptic orbits in semisimple Lie algebras, and the canonical representatives of pseudo-Hermitian symmetric elliptic orbits},  J. Math. Soc. Japan 59 (2007),  1135-1177.

\bibitem{g1} A. Gray, J. Wolf,  {\it Homogeneous spaces defined by Lie group automorphisms I,},  J. Differential Geometry 2 (1968),  77-114.

\bibitem{g2} A. Gray, J. Wolf,  {\it Homogeneous spaces defined by Lie group automorphisms II,},  J. Differential Geometry 2 (1968),  115-159.

\bibitem{h} S. Helgason, {\it Differential Geometry, Lie Groups and Symmetric Spaces,}  AMS, Providence, RI, 2001

\bibitem{kn} A. Knapp, {\it Representation theory of semisimple Lie groups}, Princeton 1986.

\bibitem{ko} T. Kobayashi, T. Yoshino, {\it Compact Clifford-Klein forms of symmetric spaces revisited}, Pure Appl. Math. Quart. 1(2005), 603-684

\bibitem{kob} T. Kobayashi, {\it Discontinuous groups and Clifford-Klein forms of pseudo-Riemannian homogeneous manifolds}, Perspectives in Mathematics, 17 (1996), 99-165.

\bibitem{kob1} T. Kobayashi, {\it Proper actions on a homogeneous space of reductive type}, Math. Ann. 285(1989), 249-263
 
\bibitem{kob2} T. Kobayashi, {\it Discrete decomposability of the restriction of $A_q(\lambda)$ with respect to reductive subgroups and its applications,} Invent. Math. 117(1994), 181-205
 
\bibitem{kow} O. Kowalski,  {\it Generalized Symmetric Spaces,}  Springer, Berlin, 1980.

\bibitem{min} A. N. Minchenko, {\it The semisimple subalgebras of exceptional Lie algebras,} Trans. Moscow Math. Soc. 67 (2006), 225-259.

\bibitem{mos} G. D. Mostow, {\it Self-adjoint groups,} Ann. of Math. 62 (1955), 44-55.

\bibitem{ow} H. Oh, D. Witte,  {\it Compact Clifford-Klein forms of homogeneous spaces of SO(2,n),}  Geom. Dedic. 89 (2002), 25-57.

\bibitem{ok} T. Okuda,  {\it Classification of Semisimple Symmetric Spaces with $SL(2, \mathbb{R})$-proper Actions,}  J. Different. Geom. 94 (2013), 301-342.

\bibitem{ov} A. L. Onishchik, E. B. Vinberg, {\it Lie Groups and Lie Algerbas III}, Springer, 2004.

\bibitem{yo} K. Yosida, {\it A theorem concerning the semisimple Lie groups,} Tohoku Math. J. 44 (1938), 81-84.

\end{thebibliography}
\end{document}